\documentclass[12pt]{amsart}
\usepackage{tikz}
\usepackage{amscd}
\usepackage{amssymb}
\allowdisplaybreaks[1] 

\title{Fundamental invariants of systems of ODEs of higher order}
\author{Boris Doubrov \and Alexandr Medvedev}

\address{Boris Doubrov, Department of Applied Mathematics and Computer Science, Belarusian State University, Nezavisimosti Ave.~4, Minsk 220030, Belarus}
\email{doubrov@bsu.by}

\address{Alexandr Medvedev, Department of Mathematics and Statistics, Masaryk University, Kotl\'a\v rsk\'a 2, Brno 61137, Czech Republic}
\email{medvedeva@math.muni.cz} 


\newcommand{\R}{\mathbb{R}}

\newcommand{\E}{\mathcal E}
\newcommand{\Sop}{\mathcal S}
\newcommand{\J}{\textbf{J}}
\newcommand{\D}{\mathcal D}

\newcommand{\cG}{\mathcal G}

\newcommand{\dd}[1]{\frac{\partial}{\partial #1}}
\newcommand{\p}{\partial}
\renewcommand{\th}{\theta}
\newcommand{\Th}{\Theta}
\newcommand{\ot}{\otimes}
\newcommand{\op}{\oplus}
\newcommand{\we}{\wedge}
\newcommand{\om}{\omega}
\newcommand{\Om}{\Omega}
\newcommand{\al}{\alpha}

\newcommand{\tfp}{\operatorname{trfp}}

\newcommand{\g}{\mathfrak{g}}
\newcommand{\ga}{\mathfrak{a}}
\newcommand{\gsl}{{\mathfrak{sl}}}
\newcommand{\sltw}{{\mathfrak{sl}(2,\R)}}
\newcommand{\gl}{\mathfrak{gl}} 
\newcommand{\gh}{\mathfrak h}

\newcommand{\glm}{{\mathfrak {gl}}(m,\R)}

\newcommand{\im}{\operatorname{im}}
\newcommand{\Hom}{\operatorname{Hom}}
\newcommand{\End}{\operatorname{End}}
\newcommand{\Inv}{\operatorname{Inv}}
\newcommand{\Id}{\operatorname{Id}}
\newcommand{\tr}{\operatorname{tr}}

\newcommand{\Ad}{\operatorname{Ad}}

\newtheorem{thm}{Theorem}
\newtheorem*{thm*}{Theorem}
\newtheorem{prop}{Proposition}
\newtheorem{lem}{Lemma}

\theoremstyle{definition}
\newtheorem{dfn}{Definition}
\newtheorem{ex}{Example}

\thanks{The second author is supported by the project CZ.1.07/2.3.00/20.0003
  of the Operational Programme Education for Competitiveness of the Ministry
  of Education, Youth and Sports of the Czech Republic.}

\begin{document}
\begin{abstract}
We find the complete set of fundamental invariants for systems of ordinary differential equations of order $\ge 4$ under the group of point transformations generalizing similar results for contact invariants of a single ODE and point invariants of systems of the second and the third order.

It turns out that starting from systems of order $(k+1)\ge 4$, the complete set of fundamental invariants is formed by $k$ generalized Wilczynski invariants coming from the linearized system and an additional invariant of degree $2$.
\end{abstract}

\maketitle

\section{Introduction}

\subsection{Notion of invariants of differential equation}
The geometry and local equivalence of ordinary differential equations has a long history.  It starts with pioneer works of Sophus Lie who introduced the notion of Lie pseudogroups and suggested a method of finding invariants of various geometric objects under the action of a certain pseudogroup.

In this paper we consider systems of ordinary differential equations viewed up to so-called pseudogroup of point transformations. In coordinate notation this means we treat systems of $m\ge 2$ equations of order $k+1\ge 4$:
\begin{equation}\label{eq1}
 \left(y^i\right)^{(k+1)}=f^i\left(\left(y^j\right)^{(s)},x\right), \quad s=0,\dots,k,
\end{equation}
up to arbitrary locally invertible changes of independent variable $x$ and dependent variables $y^i$, $1\le i\le m$. The reason for denoting the order of the system of ODEs by $(k+1)$ is purely technical and is linked with the standard 
notation for irreducible $\sltw$-modules.

By a \emph{(point) invariant of order $r$} of system~\eqref{eq1} we mean a function depending on the right hand side of~\eqref{eq1} and its partial derivatives up to order $r$ that is not changed under the action of the pseudogroup of point transformations. In the current paper we shall outline a method to construct (all) point invariants of~\eqref{eq1}, but will not actually compute them explicitly. Instead we shall work with so-called \emph{relative invariants} and find the minimal set of them that generates all other invariants (absolute and relative) via operation of invariant differentiation. We call such set of relative invariants \emph{fundamental invariants}. 

Naturally, all non-constant invariants constructed in this paper will vanish for the \emph{trivial system of equations}, i.e., the system of form~\eqref{eq1} with vanishing right hand side. As an immediate application, we get an explicit characterization of the class of \emph{trivializable systems of ODEs}, i.e., the systems that are equivalent to the trivial equation under point transformations. So, fundamental system of invariants can be thought as a minimal set of differential relations describing the orbit of the trivial equation.

The notions of relative and absolute invariants can be formalized in a rigorous way using the language of jet spaces. Namely, first we consider the jets $J^{k+1}(\R,\R^m)$ of maps from $\R$ to $\R^m$ along with natural projections $\pi_{k+1,l}\colon J^{k+1}(\R,\R^m) \to J^l(\R,\R^m)$ for any $k\ge l\ge 0$. Then the system of ODEs of order $(k+1)$ is a submanifold $\E$ of codimension $k+1$ in $J^{k+1}(\R,\R^{m})$ transversal to fibers of the projection $\pi_{k+1,k}$. Locally this is equivalent to defining the section $s\colon J^{k}(\R,\R^m)\to J^{k+1}(\R,\R^m)$ of the projection  $\pi_{k+1,k}$, and equations~\eqref{eq1} are just coordinate expressions of this section. A solution of the equation $\E$ is an arbitrary map $y\colon \R\supset U\to \R^m$ such that its $(k+1)$-st prolongation $y^{(k+1)}=\{[y]^{k+1}_x\mid x\in U\}$ belongs to~$\E$.   Here by $[y]^{k+1}_x$ we denote the $(k+1)$-jet of map $y$ at $x\in \R$. 

Recall that we can identify $J^0(\R,\R^m)$ with $\R\times \R^m=\R^{m+1}$, and any local diffeomorphism $\phi\colon  \R^{m+1}\to \R^{m+1}$ has a unique prolongation $\phi^{(k+1)}\colon J^{k+1}(\R,\R^m)\to J^{k+1}(\R,\R^m)$ compatible with prolongation operation as well as with projections $\pi_{k+1,l}$. The family of all prolongations $\phi^{(k+1)}$ is exactly the Lie pseudogroup of point transformations $\D^{k+1}(\R,\R^m)$ acting on submanifolds $\E\subset J^{k+1}(\R,\R^m)$ corresponding to the systems of ODEs of order $(k+1)$.

Denote by $\J^{k+1,r}$ the space of $r$-jets of all equation submanifolds~$E$ in $J^{k+1}(\R,\R^m)$. It is an open subset in the manifold of $r$-jets of all submanifolds of codimension $k+1$ in $J^{k+1}(\R,\R^m)$. The action of the pseudogroup $\D^{k+1}=\D^{k+1}(\R,\R^m)$ is naturally prolonged to $\J^{k+1,r}$. 

\begin{dfn}\label{dfn:1}
An \emph{(absolute) invariant of order $r$ for the systems of ODEs of order $(k+1)$} is a function on $J^{k+1,r}$ (or on an open dense $\D^{k+1}$-invariant subset therein) which is preserved by the prolonged action of $\D^{k+1}$. A \emph{relative invariant of order $r$} is a $\D^{k+1}$ equivariant map $I\colon \J^{k+1,r} \to V$ taking values in a finite-dimensional representation of the pseudogroup $\D^{k+1}$.
\end{dfn}

Let us list a number of known invariants for ODEs and systems of ODEs.
\begin{ex}
\emph{Wilczynski invariants} were first introduced by Wilczynski~\cite{wil} as a fundamental set of invariants for linear scalar ODEs on one function $y(x)$ viewed up to Lie preudogroup $(x,y)\mapsto (\lambda(x),\mu(x)y)$. The generalization of Wil\-czynski invariants to systems of linear ODEs was obtained by Se-ashi~\cite{se93}. 

Consider an arbitrary system of linear ordinary differential equations of order $k+1$:
\[
y^{(k+1)}+P_k(x)y^{(k)}+\dots+P_0(x)y(x)=0,
\]
where $y(x)$ is an $\R^m$-valued vector function. It can always be brought to the so-called canonical \emph{Laguerre--Forsyth form} of these equations defined by conditions $P_{k}=0$ and $\tr P_{k-1}=0$. Then the following expressions:
\begin{equation}\label{wil_inv}
\Theta_r = \sum_{j=1}^{r-1} (-1)^{j+1}
\frac{(2r-j-1)!(k-r+j)!}{(r-j)!(j-1)!} P_{k-r+j}^{(j-1)}, \quad
r=2,\dots, k+1.
\end{equation}
are the $\End(\R^m)$-valued relative invariants of the initial system, where each invariant $\Th_r$ has a degree $r$. Note that the first non-trivial Wilczynski invariant has degree $2$ and is trace-free. In particular, it vanishes identically in the scalar case, but is a non-trivial invariant in case of systems of $m\ge 2$ linear equations.

\emph{Generalized Wilczysnki invariants} $W_r$, $r=2,\dots,k+1$ of system~\eqref{eq1} are defined as invariants $\Theta_r$ evaluated at the linearization of the system. Formally they are obtained from~\eqref{wil_inv} by substituting $P_r(x)$ by matrices $-\left(\frac{\p f^i}{\p (y^j)^{(r)}}\right)$ and the usual derivative by the total derivative:
\[
\frac{d}{dx} = \dd{x} + y^{(1)}\dd{y}+ \dots+y^{(k-1)}\dd{y^{(k-2)}}+f\dd{y^{(k-1)}}.
\]

They turn out to be relative invariants of the non-linear system of ODEs in the sense of Definition~\ref{dfn:1}. See~\cite{dou08} for more details. 
\end{ex}

\begin{ex}
One of the simplest cases, when the complete set of fundamental invariants is known, is the class of second order ODEs $y''=f(x,y,y')$ viewed up to the pseudogroup of all point transformations. The fundamental invariants in this case
was computed by Tresse~\cite{tresse96} at the end of the 19-th century. The simplest non-trivial relative invariants appear only at order $4$: 
\begin{align*}
    I_1 &= f_{1111} = \frac{\partial^4 f}{\left(\partial y'\right)^4};\\
    I_2 &= \frac16f_{11xx}-\frac16f_1f_{11x}-\frac23f_{01x}+
           \frac23f_1f_{01}+f_{00}-\frac12f_0f_{11}.
\end{align*}
Here we use the standard classical notation used in case of a single ODE of order $k$: the subscript $i=0,\dots,k-1$ means the partial derivative with respect to $y^{(i)}$ and the subscript $x$ means the total derivative.

The geometry lying behind this equation was explored by E.~Cartan~\cite{car24}, who
associated the the canonical coframe with this equation and proved
that all its invariants can be derived by covariant
differentiation from the above two relative invariants. In particular, the general equation is equivalent to the trivial one
if and only if $I_1=I_2=0$. See also~\cite{ns03} for the interpretation of these invariants in terms of the associated Fefferman metric.
\end{ex}

\begin{ex}
In the case of a single ODE of the order $\ge3$ there are two kinds of pseudogroups (and invariants) typically considered in the applications. Namely, the largest pseudogroup acting on ODEs of the fixed order (and preserving this order) is a so-called \emph{pseudogroup of contact transformations}. These are the transformations of $J^k(\R,\R)$ of the form $\phi^{(k-1)}$, where $\phi$ is a local transformation of $J^1(\R,\R)$ preserving the natural contact structure on it.

The equivalence problem of the 3rd order ODEs $y'''=f(x,y,y',y'')$ up to contact transformations was studied by S.-S. Chern~\cite{chern40}, who found the following set of fundamental invariants in this case:
\begin{align*}
    I_1 &= f_{2222} = \frac{\partial^4 f}{\left(\partial y''\right)^4};\\
    W &= -f_0-\frac{1}{3}f_1f_2-\frac{2}{27}f_2^3+\frac{1}{2}f_{1x}+
           \frac{1}{3}f_2f_{2x}-\frac{1}{6}f_{2xx}.
\end{align*}
The second of these invariants was first found by W\"unschmann~\cite{wun05} back in 1905 and is usually called \emph{W\"unschmann invariant} in his honor.

The point geometry of the 3rd order ODEs is quite different and was studied by Elie Cartan~\cite{car41}. It turns out that a given equation $y'''=f(x,y,y',y'')$ is equivalent to the trivial one up to point transformations if and only if the following conditions are satisfied:
\begin{align*}
& W = f_{222}  = f_{22}^2+ 6f_{122}+2f_2f_{222} = 0;\\
& C=f_{11}+2W_{2}-2f_{02}+\tfrac{2}{3}f_2f_{12}+2f_{22}\left(\tfrac{1}{3}f_{2x}-\tfrac{2}{9}f_2^2-f_1\right) = 0,
\end{align*}
where $W$ is the above W\"unschmann invariant and $C$ is a so-called Cartan invariant~\cite{car41}. 

Note that the W\"unschmann invariant $W$ for scalar 3rd order ODE is exactly the lowest degree generalized Wilczynski invariant in case of scalar ODEs.
\end{ex}

In the following examples and further in the paper we follow the Einstein summation convention and denote by $\tfp$ the trace free part of an arbitrary tensor. 

\begin{ex}
Systems of  the second order were studied in the work of Mark Fels~\cite{fels95}. Starting from an arbitrary system of the 2nd order ODEs:
\[
(y^i)'' = f^i(x,y^j,(y^k)'),\quad i=1,\dots,m,
\]
he constructs an absolute parallelism solving the equivalence problem as well as two fundamental invariants. They appear in the degree $2$ and $3$ and are equal to the following tensors:
\begin{align*}
(W_2)^i_j &=\tfp \left( \frac{\p f^i}{\p y^j}-\frac12 \frac{d}{dx}\left(\frac{\p f^i}{\p p^j}\right)+
\frac14\frac{\p f^i}{\p p^r}\frac{\p f^r}{\p p^j}\right);\\
(I_3)^i_{jkl} &= \tfp \left( \frac{\p^3 f^i}{\p p^j\p p^k\p p^l} \right),
\end{align*}
where by $p^i$ we denote $(y^i)'$ , $i=1,\dots,m$, which together with $x$ and $y^j$ form local coordinates on $J^1(\R,\R^m)$, and, as above, $\frac{d}{dx}$ denotes the operator of total derivative with respect to $x$. 

Note that the above invariant $W_2$ is the simplest non-trivial example of a generalized Wilczynski invariant in case of systems of ODEs. 
\end{ex}

We note that in case of systems of ordinary differential equations the classes of point and contact transformations  coincide and we shall always consider the pseudogroup of point transformations in case of systems of ODEs without mentioning this explicitly.

\begin{ex}
Fundamental invariants for systems of the third order were computed in the works of Alexandr Medvedev~\cite{medv12}. Then for the systems of $m$ ODEs of the 3rd order the following tensors form the minimal set of fundamental invariants:
\begin{enumerate}
\item two generalized Wilczynski invariants $W_2$ and $W_3$:
\begin{align*} 
\left(W_2\right)^i_j =& \tfp \left(\frac{\p f^i}{\p p^j} -
\frac{d}{dx} \frac{\p f^i}{\p q^j} +
\frac{1}{3} \frac{\p f^i}{\p q^k}\frac{\p f^k}{\p q^j} \right), 
\\
\left(W_3\right)^i_j =&
\frac{\partial f^i}{\partial y^j} +\frac{1}{3} \frac{\partial f^i}{\partial q^k}\frac{\partial f^k}{\partial p^j} - \frac{d}{dx} \frac{\partial f^i}{\partial
p^j} + \frac{2}{3}\frac{d^2}{dx^2} \frac{\partial f^i}{\partial
q^j} + \frac{2}{27}\left(\frac{\partial f^i}{\partial q^j}\right)^3 \\ 
&-\frac{4}{9}
\frac{\partial f^i}{\partial q^k}\frac{d}{dx}\frac{\partial f^k}{\partial q^j} 
 -\frac{2}{9} \frac{d}{dx}\left(\frac{\partial f^i}{\partial q^k}\right)\frac{\partial f^k}{\partial q^j} - 2\delta^i_j\frac{d}{dx}H^x;
\end{align*}
\item two additional invariants of the degree $2$ and $4$:
\begin{align*}
\left(I_2\right)^i_{jk} =& \tfp\left(\frac{\p^2 f^i}{\p q^j \p q^k} \right), 
\\
 \left(I_4\right)_{jk} =& - \frac{\p H_k^{-1}}{\p p_j}+\frac{\p}{\p q_j}\frac{\p}{\p q_k}H^x - \frac{\p}{\p q_k}\frac{d}{dx}H_j^{-1} \\ 
&\qquad\qquad\qquad - \frac{\p }{\p q^k}\left(H_l^{-1}\frac{\p f^l}{\p q^j}\right)+2H_j^{-1} H_k^{-1},
\end{align*}  
where 
\begin{align*}
H_j^{-1}&=\frac{1}{6(m+1)}\left(\frac{\p^2 f^i}{\p q^i \p q^j}
 \right),\\
H^x &=-\frac{1}{4m}\left(\frac{\p f^i}{\p p^i} -
\frac{d}{dx} \frac{\p f^i}{\p q^i} +
\frac{1}{3} \frac{\p f^i}{\p q^k}\frac{\p f^k}{\p q^i} \right). 
\end{align*}
\end{enumerate} 
\end{ex}

\begin{ex}
Fundamental invariants for scalar ODEs of order $4$ where computed by Robert Bryant~\cite{bry4ord} and for orders $\ge 5$ by Boris Doubrov~\cite{dou01}. It turns out that apart from generalized Wilczynski invariants there are only 2 additional invariants for ODEs of order $4$, $5$, $6$ and three additional invariants for orders $\ge 7$. 
They are given by:
\begin{itemize}
\item invariant $I_3=f_{333}$ for 4th order ODE  and invariant $I_2=f_{k,k}$ for $(k+1)$-th order ODE, $k+1\ge 5$;
\item extra invariants:\\
$k+1=4$: $J_4= f_{233}+\frac16f_{33}^2+\frac98f_3f_{333}+\frac34f_{333x}$;\\
$k+1=5$: $J_6 = f_{234}-\frac23f_{333}-\frac12f_{34}^2\mod I_2, W_3$;\\
$k+1\ge 6$: $J_3= f_{k,k-1} \mod I_2$;\\
$k+1\ge 7$: $J_4=f_{k-1,k-1} \mod I_2,J_3,W_3$.
\end{itemize}
\end{ex}

To summarize the above examples, we see that the fundamental invariants of ODEs were computed up to now for a single ODE of an arbitrary order (under the pseudogroup of contact transformations) and for systems of ODEs of order~$2$ and~$3$. For systems of ODEs of higher order only a part of invariants (namely, generalized Wilczynski invariants) was known.

This paper closes this gap and computes the complete set of fundamental invariants in all remaining cases, namely for systems of ODEs of order $\ge 4$.  
\begin{thm*}
The following relative invariants form a fundamental set for systems of $m\ge 2$ ODEs of order $k+1 \ge 4$:
\begin{enumerate}
\item generalized Wilczynski invariants $W_r$ of the degree $r$, $2\le r \le k+1$;
\item one additional invariant $I_2$ of the degree $2$:
\[
(I_2)^i_{jl}=  \frac{\p^2 f^i }{\p y_k^j \p y_k^l}.
\]
\end{enumerate}
\end{thm*} 

The paper is organized as follows. In Section~\ref{sec:cc} we recall the construction of \emph{the normal Cartan connection} from~\cite{dkm} and show that the fundamental set of invariants is described by the cohomology space $H^2_{+}(\g_{-},\g)$, where $\g$ is a symmetry algebra of the trivial system of ODEs equipped with an appropriate grading. It is defined in Subsection~\ref{ss:acc}. 

We compute this cohomology space in Section~\ref{sec:cohom} using the Serre--Hochschild spectral sequence. However, not all non-zero elements of this cohomology space lead to non-trivial fundamental invariants. It turns out that some
of the corresponding fundamental invariants vanish identically. This can be explained as follows. A part of $H^2_{+}(\g_{-},\g)$ corresponds to the invariants of the underlying non-holonomic distribution. In our case this non-holonomic distribution is a contact distribution on the jet space $J^k(\R,\R^m)$. It is a flat distribution of type $\g_{-}$, where $\g_{-}$ can also be viewed as the Tanaka symbol of the contact distribution. 

To find out which fundamental invariants are non-trivial, we perform explicit coordinate computation of a part of the normal Cartan connection in Section~\ref{sec:param} and show that its curvature satisfies certain additional linear relations. Projecting these relations to $H^2_{+}(\g_{-},\g)$, we find out that a large part of this cohomology space produces identically vanishing invariants.  This way we are also able to show that the above invariant of degree 2 is the only additional fundamental invariant that, together with generalized Wilczynski invariants, forms the complete set of fundamental invariants.

Note that this phenomenon, that part of the positive cohomology corresponds to trivial invariants, does not happen in the theory of parabolic geometries~\cite{capslovak}. But is was already observed in earlier papers on the invariants of differential equations~\cite{dou01,medv12}, where only a part of the cohomology space $H^2_{+}(\g_{-},\g)$ corresponds to non-trivial invariants.

Finally, in Section~\ref{coh_M} we show that the \emph{effective part} (i.e., the part corresponding to non-trivial invariants) of the cohomology space $H^2_{+}(\g_{-},\g)$ coincides with the kernel of the natural linear map
\[
\gamma\colon H^2(\g_{-},\g)\to H^2(\g_{-},\bar\g),
\] 
where $\bar\g$ is the (infinite-dimensional) Tanaka prolongation of $\g_{-}$. The latter cohomology space was computed by T.~Morimoto~\cite{mor88}, and the map $\gamma$ is well-defined, since $\g$ is a finite-dimensional subalgebra of $\bar\g$. However, we do not give any conceptual proof of this fact in this paper and consider it merely as a hint that helps us identifying the effective part of the 2nd cohomology space. 

\section{Normal Cartan connection for systems of ODEs}
\label{sec:cc}

One of the main techniques for computing the fundamental invariants for differential equations is Cartan's equivalence method and its further generalization of N.~Tanaka in the context of so-called nilpotent differential geometry. The main advantage of Tanaka's approach is the adaptation of all constructions to the underlying non-holonomic vector distribution (the contact distribution on the jet spaces), the powerful algebraic techniques for constructing a normal Cartan connection (instead of much weaker absolute parallelism structures) and the way to describe the principal part of the curvature (in our terminology this is exactly the set of fundamental invariants) via cohomology of finite-dimensional graded Lie algebras. In this section we outline how the geometry of systems of ODEs fits into the framework of nilpotent differential geometry. Further details and references can be found in~\cite{dkm}.

\subsection{System of ODEs as a filtered manifold}
Let us briefly describe the geometric structures defined by system~\eqref{eq1}. Consider a smooth manifold $M=\R^{m+1}$ with coordinates $(x,y^i)$, where $i=1,\dots,m$. Let $J^{k+1}=J^{k+1}(\R,\R^m)$ be the space of $(k+1)$-jets of smooth maps from $\R$ to $\R^m$. We use the standard local coordinate system in $J^{k+1}(\R,\R^m)$:
\[
(x,y^i_r), \quad 0 \le r \le k+1, 1\le i \le m,
\]
where we assume that $y^i_a$ has the meaning of $a$-th derivative of $y^i(x)$.

System of ODEs~\eqref{eq1} defines a submanifold $\E$ in $J^{k+1}(\R,\R^m)$ by:
\[ 
y^i_{k+1} = f^i\left(y^j_r,x\right) ,\quad i=1,\dots,m, r=0,\dots,k.
\]
The projection $\pi_{k+1,k}\colon J^{k+1}\to J^k$ establishes a (local) diffeomorphism $\E$ with $J^{k}$. With every system of ODEs we associate a pair of distributions on $\E$:
\begin{align*}
E&=\left\langle \frac{\partial}{\partial x}+\sum_{r=1}^k y_r^i\frac{\partial}{\partial y_{r-1}^i} + f^i\frac{\partial}{\partial y_{k}^i}\right\rangle,\\
V&=\left\langle \frac{\partial}{\partial y^i_k}\right\rangle,
\end{align*}
where $i=1,\dots,m$. The distribution $C^{-1}=E\oplus V$ is mapped to the standard contact distribution on $J^k$ under local diffeomorphism $\pi_{k+1,k}\colon \E\to J^k$. In particular, it is bracket generating, and its weak derived series defines a filtration of the tangent bundle $T\E$:
\[
C^{-1}\subset C^{-2} \subset \dots \subset C^{-k-1} = T\E ,
\]
where $C^{-i-1}=C^{-i}+[C^{-i},C^{-1}]$. 

We call an arbitrary coframe $\om_x, \om^i_{-r}$, $1\le i\le m, 1\le r\le k+1$, on $\E$ \emph{adapted to the equation~\eqref{eq1}}, if
\begin{enumerate}
\item[(a)] the annihilator of forms $\om_{-r-1}^i,\dots,\om_{-k-1}^i$ is equal to $C^{-r}$ for all $r=1,\dots,k$;
\item[(b)] the annihilator of forms $\om_{-1}^i,\dots,\om_{-k-1}^i$ is equal to $E$;
\item[(c)] the annihilator of $\om_x,\om_{-2}^i,\dots,\om_{-k-1}^i$ is equal to $F$.
\end{enumerate}

We call an adapted coframe \emph{regular}, if in addition it satisfies the condition: 
\begin{enumerate}
\item[(d)] $d\om_{-r}^i + \om_x \wedge \om_{-r+1}^i = 0 \mod \langle \om_{-r}^j,\dots,\om^j_{-1} \rangle$ for all $2\le r\le k+1$.
\end{enumerate}

\subsection{Adapted Cartan connections}\label{ss:acc}
Let $\g$ be the symmetry algebra of the trivial system of $m$ ODEs of order $(k+1)$. In the sequel we assume that $k\ge3$, $m\ge2$, that is we consider only systems of ordinary differential equations of an order $\ge 4$. Under these conditions the Lie algebra $\g$ is isomorphic to the semidirect product of a reductive Lie algebra $\ga=\gsl(2,\R)\times \gl(m,\R)$ and an abelian ideal $V=V_k\otimes W$, where $V_k$ is an irreducible $\gsl(2,\R)$-module isomorphic to $S^k(\R^2)$ and $W=\R^m$ is the standard $\gl(m,\R)$-module.

Let us fix a basis of $\g$. Let:
\[ 
x=\begin{pmatrix} 0 & 0\\ 1 & 0 \end{pmatrix},\quad
y=\begin{pmatrix} 0 & 1\\ 0 & 0 \end{pmatrix},\quad
h=\begin{pmatrix} -1 & 0\\ 0 & 1 \end{pmatrix}
\]
be a basis for $\gsl(2,\R)$. Fix a basis of $\sltw$-module $V_k$ consisting of elements $v^i=f_2^{k-i}f_1^i/i!$, where $f_1,f_2$ is the standard basis in $\R^2$. We denote by $\{e_1,\dots,e_m\}$ and $\{e^i_j\}$ the standard bases of $\R^m$ and $\gl(m,\R)$ respectively. (That is we have $e^i_je_k=\delta_{ik}e_j$.)

The degrees of elements in the Lie algebra $\g$ are defined as follows:
\begin{align*}
\g_1&=\R y,\\
\g_0&=\R h\oplus\gl(m,\R),\\
\g_{-1}&=\R x\oplus\R v^k\otimes W,\\
\g_{-i}&=\R v^{k+1-i}\otimes W, \quad i=2,\dots,k+1,
\end{align*}
and $\g_n=\{0\}$ for all other $n\in\mathbb Z$.
As we see, the negative part $\g_{-}$ of $\g$ is equal to $\R x\oplus V$. We denote the non-negative part of the Lie algebra~$\g$ by $\gh$.

Globally, we can define a Lie group $G$ as a semidirect product of $SL(2,\R)\times GL(m,\R)$ and the commutative group $V=V_k\otimes \R^m$. Let $H$ be the direct product of subgroup $ST(2,\R)$ of all lower-triangular matrices in
$SL(2,\R)$ and $GL(m,\R)$. Then $G/H$ can be identified with the trivial equation $\E_0\subset J^{k+1}(\R,\R^m)$.

Our next goal is to construct a Cartan connection on $\E$, modeled on the homogeneous space $G/H$, that will be naturally associated with the equation~\eqref{eq1}. Such Cartan connection consists of a principal $H$-bundle $\pi \colon \cG\to \E$ and the $\g$-valued differential form $\om$ on $\cG$ such that
\begin{enumerate}
\item $\om(X^*)=X$ for all fundamental vector fields $X^*$ on $\cG$, $X\in\gh$;
\item $R_h^*\om = \Ad h^{-1}\om$ for all $h\in H$; 
\item $\om$ defines an absolute parallelism on $\cG$.
\end{enumerate}

Any $\g$-valued form $\om$ can be written as
 \[
 \omega =\sum_{r=1}^{k+1}\omega^i_{-r} (v^{k+1-r}\otimes e_i) + \omega_x x  + \omega_h h +
\omega^i_j e^j_i + \omega_y y ,
\]
where $\omega^i_{-r}$, $\omega^i_j$, $\omega_x$, $\omega_h$, $\omega_y$ are 1-fomrs on $\cG$.

We say that a Cartan connection $\om$ on a principal $H$-bundle $\pi\colon \cG\to \E$ is \emph{adapted to the equation~\eqref{eq1}}, if for any local section $s$ of $\pi$ the set $\{s^*\om_x, s^*\om^i_{-r}\}$ is an adapted coframe on $\E$. The Cartan connection $\omega$ is said to be \emph{regular}, if the above coframe is also regular.
It is easy to see that this definition does not depend on the choice of the local section~$s$.

Denote by $C^q(\g_{-},\g)$ the space of all $q$-cochains on $\g_{-}$ with values in $\g$. Any Cartan connection $\om$ modeled by the homogeneous space $G/H$ determines the curvature tensor $\Om=d\om+1/2[\om,\om]$ on $\cG$ and the curvature function $c\colon \cG\to C^2(\g_{-},\g)$, where
\begin{equation*}
  c_p(u,v)=\Om_p(\om^{-1}_p(u),\om^{-1}_p(v))\quad\text{for all  }u,v\in\g_{-},\ p\in \cG.
\end{equation*}
This function satisfies the condition
\begin{equation}\label{eq:3}
c(ph)=h^{-1}.c(p)\quad\text{for all }  h\in H,p\in \cG,
\end{equation}
where $H$ acts on $C^2(\g_{-},\g)$ in the natural way.

Since $\g_{-}$ and $\g$ are graded, all spaces $C^q(\g_{-},\g)$ inherit the gradation
\begin{gather*}
  C^q(\g_{-},\g)=\sum_r C^q_r(\g_{-},\g),\\
  \intertext{where}
  C^q_r(\g_{-},\g)=\{ \alpha\in C^q(\g_{-},\g)\mid \alpha(\g_{i_1},\dots,\g_{i_q})\subset \g_{i_1+\dots+i_q+r}\}.
\end{gather*}
The standard cochain differential
\begin{equation*}
 \partial\colon C^q(\g_{-},\g)\to C^{q+1}(\g_{-},\g)
\end{equation*}
preserves this gradation. Decompose the curvature function $c$ to the sum $c=\sum_r c_r$, where each $c_r$ takes values in $C^2_r(\g_{-},\g)$. It is easy to see that regularity of the Cartan connection implies that $c_r=0$ for all $r\le 0$. In other words,  the curvature function of a regular Cartan connection takes values in $C^2_{+}(\g_{-},\g)$.

The space $C^2(\g_{-},\g)$ admits also another $\gh$-invariant decomposition coming from the decomposition of $\g$ into the sum of the reductive part $\ga=\gsl(2,\R)\times \gl(m,\R)$ and the abelian ideal $V$. We note that
\[
\wedge^2 \g_{-} = \wedge^2 (\R x + V) \cong \R x \otimes V + \wedge^2 V.
\]
Therefore, $C^2(\g_{-},\g)$ is naturally decomposed into four subspaces:
\begin{multline*}
C^2(\g_{-},\g) = \Hom(\R x \otimes V, \ga) + \Hom(\R x \otimes V, V) \\ 
+ \Hom(\wedge^2 V, \ga) + \Hom(\wedge^2 V, V).
\end{multline*}
It is easy to see that this decomposition is $\gh$-invariant and is compatible with the above grading. 

The curvature function $c$ of a regular Cartan connection is decomposed accordingly into four summands. We shall be mainly interested in the second and fourth components in this decomposition and shall call them the $\Hom(\R x \otimes V, V)$ and $\Hom(\wedge^2 V, V)$ parts of the curvature (function). 

\subsection{Harmonic theory on on the cochain complex of the symbol algebra}
Generally speaking, there are many regular Cartan connections adapted to a given equation~\eqref{eq1}. The basic idea of choosing a unique one among them is to add linear conditions on structure function. Finding these conditions is not easy since they should guarantee the existence and uniqueness of the required Cartan connection and at the same time they must be invariant with respect to the action of $H$ on $C^2(\g_{-},\g)$ because of property~\eqref{eq:3}
of the curvature function $c$. Fortunately, as it was shown by T.~Morimoto~\cite{mor89,mor93}, they can be derived from the ``harmonic theory'' on our symbol Lie algebra $\g$.

First, we fix a scalar product $(,)$ on $\g$ such that vectors $x,y,h,v^i\otimes e_j, e^i_j$ form an orthogonal basis and
\begin{gather*}
  \langle e^i_j,e^i_j\rangle=1,\quad (v^i\otimes e_j, v^i\otimes e_j) = (k-i)!/i!, \quad 0\le i\le k;\\
  (x,x)=(y,y)=1, (h,h)=2.
\end{gather*}
This metric $(,)$ is chosen in such a way that
\begin{enumerate}
\item all spaces $\g_i$ are mutually orthogonal;
\item $(S,T) = \tr {}^tST$ for all $S,T\in\sltw\times \gl(m,\R)$;
\item $(Su,v)=(u,{}^tSv)$ for all   $S\in\sltw\times \gl(m,\R)$, $u,v\in V$, so that the transposition with respect to
  this metric on $V$ agrees with standard matrix transpositions in $\sltw$ and $\gl(m,\R)$.
\end{enumerate}

Then we extend this metric to the spaces $C^q(\g_{-},\g)$ in the standard way and denote by
\begin{equation*}
\partial^*\colon C^{q+1}(\g_{-},\g)\to C^q(\g_{-},\g)
\end{equation*}
the operator adjoint to the cochain differential $\partial$.

Finally, we add one more condition on our Cartan connection adapted to equation~\eqref{eq1}.
\begin{prop}[\cite{dkm,mor89}] Among all Cartan connections adapted to equation~\eqref{eq1} there exists a unique (up to isomorphism) Cartan connection whose structure function is co-closed, i.e., $\partial^* c = 0$.
\end{prop}
In the sequel we call this Cartan connection \emph{a normal Cartan connection associated with equation~\eqref{eq1}} and denote it by $\om_\E$.

\subsection{Fundamental invariants}
Normal Cartan connections give also an algorithm of constructing invariants of ordinary differential equations. Let $\om_{\E}\colon T\cG\to\g$ be the normal Cartan connection associated with the equation $\E$. Since $\om_\E$ defines an absolute parallelism on $\cG$, we see that the algebra of its invariants is generated by coefficients of its structure function and their covariant derivatives. Using the special properties of the canonical Cartan connection and in particular its deep relation with harmonic theory on the symbol algebra $\g$, we may reduce the number of generators of the algebra of invariants of $\om_\E$.

\begin{prop}[\cite{dkm}]
The algebra of invariants of the canonical Cartan connection $\om_\E$ is generated by the coefficients of the harmonic part of the structure function $c$ of $\om_\E$ and its covariant derivatives. In particular, the curvature function $c$ vanishes if and only if its harmonic part vanishes.
\end{prop}

As the harmonic part of the cochain complex $C(\g_{-},\g)$ is naturally isomorphic to the corresponding Lie algebra cohomology $H(\g_{-},\g)$, the number of fundamental invariants and their degrees can be determined by computing the cohomology spaces $H^2_{+}(\g_{-},\g)$. This will be done in the next section.

\section{Lie algebra cohomology related to systems of ODEs of higher order}
\label{sec:cohom}
Consider the cohomology spaces $H^q(\g_{-},\g)$. They can be naturally supplied with the grading:
\begin{gather*}
H^q(\g_{-},\g)=\bigoplus_{r\in\mathbb Z} H^q_r(\g_{-},\g),\\
\intertext{where}
H^q_r(\g_{-},\g)=\{[c]\in H^q(\g_{-},\g)\mid
c(\g_{i_1},\dots,\g_{i_q})\subset \g_{i_1+\dots+i_q+r}\}.
\end{gather*}

As mentioned above, the fundamental invariants of a system of the $(k+1)$-th order 
ODEs are described by the positive part of the second cohomology space $H^2(\g_{-},\g)$. Below
we compute this space by means of the Serre--Hochschild spectral sequence, determined by the subalgebra $V$ of $\g_{-}$. 

Recall~(see~\cite{hoch-serre, fuks}) that the Serre--Hochschild spectral sequence is one of the main technical tools for computing cohomology $H(\mathfrak{l}, A)$ of an arbitrary Lie algebra $\mathfrak{l}$ with coefficients in an $\mathfrak{l}$-module $A$ in case when the Lie algebra $\mathfrak{l}$ has a non-trivial ideal $\mathfrak{l}_0$. In this case the second term $E_2$ of this spectral sequence is equal to $E_2^{p,q}=H^p(\mathfrak{l}/\mathfrak{l_0}, H^q(\mathfrak{l_0}, A))$.  

In our case we can build the Serre-Hochschild spectral sequence taking $V$ as an ideal of $\g_{-}$. Then the second term $E_2$ of the Serre-Hochschild spectral sequence computing $H(\g_{-}, \g)$ has the form: $E_2=\bigoplus_{p,q}E_2^{p,q},$ where
\[
E^{p,q}_2=H^p(\R x, H^q(V,\g)),\quad p,q\ge0.
\]

We immediately get the following result regarding the structure of $H^2(\g_{-},\g)$:
\begin{prop}\label{l1}
 The second cohomology space $H^2(\g_{-},\g)$ is naturally isomorphic with the subspace $E_2^{1,1}\oplus E_2^{0,2}$ of the  Serre-Hochschild spectral sequence determined by the ideal $V\subset
  \g_{-}$. 

  Moreover, we have
  \begin{align*}
    E_2^{1,1}&=H^1(\R x,H^1(V,\g)),\\
    E_2^{0,2}&=H^0(\R x,H^2(V,\g))=\Inv_x H^2(V,\g).
  \end{align*}
\end{prop}
\begin{proof}
Since the Lie algebra $\R x$ is one-dimensional, we see that
$E_2^{p,q}=\{0\}$ for all $p>1$. Therefore, the differential 
\[
d_2^{p,q}\colon E_2^{p,q}\to E_2^{p+2,q-1}
\]
is trivial and the spectral sequence stabilizes in the second term.
\end{proof}

Using the following observation, the computation of $H^2(\g_{-},\g)$ can be reduced essentially to the decomposition of $\sltw$-modules $H^1(V,\g)$ and $H^2(V,\g)$ into sums of irreducible submodules. 
\begin{lem}
  Let $V_q$ be a $(q+1)$-dimensional irreducible $\sltw$-module. The space $H^p(\R x,V_q)$ is trivial for $p\ge2$ and is  one-dimensional for $p=0,1$. 

  Let $v_0$ and $v_q$ be the highest and the
  lowest weight vectors of $V_q$ (that is $h.v_0=qv_0$ and
  $h.v_q=-qv_p$). Then $H^0(\R x,V_q)$ is generated by $v_0$, and
  $H^1(\R x,V_p)$ is generated by $[\alpha\colon x\to v_p]$.
\end{lem}
\begin{proof}
  Immediately follows from the explicit description of the structure of irreducible $\sltw$-modules.
\end{proof}

Let us identify $\ga$ with the subalgebra of $\gl(V)$ corresponding to the action of
$\sltw\times \glm$ on $V$. Then the cohomology spaces $H^q(V,\g)$ can be described via 
the classical Spencer cohomology spaces determined by the
subalgebra $\ga\subset \gl(V)$. Recall that the Spencer operator $\Sop^q$ is defined as:
\begin{gather*}
\Sop^q\colon \Hom(\wedge^q V,\ga)\to \Hom(\wedge^{q+1} V,V),\\
\Sop^q(\phi)(v_1\wedge v_2\wedge \dots \wedge v_{q+1})=
\sum_{i=1}^{q+1}(-1)^i\phi(v_1\wedge \dots \wedge \hat{v_i}\wedge
\dots \wedge v_{q+1})v_i.
\end{gather*}

\begin{lem}\label{l2} We have $H^0(V,\g)=V$ and
\[
H^q(V,\g)=\ker \Sop^q\oplus \Hom(\wedge^q V, V)/\im \Sop^{q-1}
\]
for all $q\ge1$. 
\end{lem} 
\begin{proof}
  Indeed, let us represent an arbitrary cocycle $c\in C^q(V,\g)$ as
  $c=c_{\ga}+c_V$, where $c_{\ga}\in \Hom(\wedge^q V,\ga)$ and $c_V\in
  \Hom(\wedge^q V,V)$.  Since $V$ is commutative Lie algebra, we have 
  \[
  (\partial c)=\Sop^q(c_{\ga})\in \Hom(\wedge^{q+1}V,V).
  \]
  This immediately implies the statement of the lemma. 
\end{proof}

For $q=1,2$ the mappings $\Sop^q$ can be described explicitly.
\begin{lem}\ \par  
  The operator $\Sop^1$ is injective if $m\ge 2$ and $k\ge 3$ .
\end{lem}  
\begin{proof}
  Let us note that $\ker \Sop^1$ is precisely the first prolongation
  $\ga^{(1)}$ of the subalgebra $\ga\subset \gl(V)$. Suppose that
  $\ga^{(1)}\ne\{0\}$.  Then the algebra $V+\ga+\sum_{i=1}^\infty
  \ga^{(i)}$ is an irreducible graded Lie algebra of depth $\ge 2$
  (see~\cite{kn65}). Then from~\cite[Lemma 7.3]{kn65} it follows
  that the difference between the highest and the lowest weights of
  $\ga$-module $V$ is equal to the sum of the highest roots of $\sltw$ and $\glm$. 
This is not possible under the assumptions of the lemma. Therefore, $\ker \Sop^1=\ga^{(1)}=\{0\}$.
\end{proof}

\begin{lem}
The operator $\Sop^2$ is injective for $m\ge3,k\ge3$ and $m=2,k\ge4$.
\end{lem}
\begin{proof}
First we prove that $\ker \Sop^2=0$ for $m\ge3,k\ge3$.

 Let $\alpha$ be an arbitrary element of $\ker \Sop^2$. Put
 \[
\alpha_{ij}(w_1,w_2)=\alpha(v^i\ot w_1,v^j\ot w_2)\in \ga. 
\] 
Let us show that $\alpha_{ij}=0$ for all $i,j\ge 2$. Indeed, we have 
  \[
  \alpha_{ij}(w_1,w_2)v^0\ot w_3-\alpha_{0j}(w_3,w_2)v^i\ot w_1+\alpha_{0i}(w_3,w_1)v^j\ot w_2=0.
  \]
  But for any element $X\in \ga$
  \[Xv^i\ot w\in
  \langle v^{i-1}\ot w,v^{i}\ot W,v^{i+1}\ot w\rangle. \]
    Hence,
  $\alpha_{ij}(w_1,w_2)v^0\ot w_3=0$ for every vector $w_3$ which is not lying in the linear span of $w_1$ and $w_2$. Therefore 
  \begin{equation}\label{eq2} \alpha_{ij}(w_1,w_2)\in \langle x, h,z \rangle,
  \end{equation} where $z$ lies in the center of~$\glm$.
  Similarly, 
   \[
  \alpha_{ij}(w_1,w_2)v^1\ot w_3-\alpha_{1 j}(w_3,w_2)v^i\ot w_1+\alpha_{1 i}(w_3,w_1)v^j\ot w_2=0.\]
   But from~\eqref{eq2} we see that 
  \[\alpha_{ij}(w_1,w_2)v^1\ot w_3\in (\R v^0\op \R v^1)\ot w_3 .\]
  Therefore, $\alpha_{ij}(w_1,w_2)v^1\ot w_3=0$ for every vector $w_3$ which is not lying in the linear span of $w_1$ and $w_2$. This is possible only if $\alpha_{ij}=0$.  In the same way we can prove that
  $\alpha_{ij}=0$ for all $i,j\le k-2$.
  
 Consider now the following subspace $Q\subset \wedge^2 V$:
  \[
  Q=\{w\in \wedge^2 V\mid \alpha(w)=0,\quad\alpha\in \ker \Sop^2\}.
  \]
  It is clear that $Q$ is a submodule of the $\sltw\times\glm$-module
  $\wedge^2 V$. As we have just proved, $(v^i\ot W)\wedge (v^j\ot W)\subset Q$
  for all pairs of $i$ and $j$ such that $i,j\ge 2$ or $i,j\le k-2$. Hence, $Q$ contains also the submodule generated by these elements. 
  
  We claim that these elements generate whole $\we^2V$. It is sufficient to prove that if $(v^i\ot W)\wedge (v^j\ot W)\subset Q$ for all pairs of $i$ and $j$ such that $i,j\ge l+1$ and $i,j\le l$ then $(v^i\ot W)\wedge (v^j\ot W)\subset Q$ for all pairs of $i$ and $j$ such that $i,j\ge l-1$ and $i,j\le l+1$. First,  consider an element $ v^{l+1} \ot m_1 \we v^{l+j}\ot m_2$ where $j\ge2$. Then after the action of $x$ on this element we get that $ v^{l} \ot m_1 \we v^{l+j}\ot m_2\in Q.$ Similarly, using the action of the element $y$ on $ v^{l} \ot m_1 \we v^{l-j+1}\ot m_2$ we get $ v^{l+1} \ot m_1 \we v^{l-j+1}\ot m_2\in Q$. 
  
  The only one type  of elements for which we don't know yet if they belong to $Q$ are elements of the form $ v^{l} \ot w_1 \we v^{l+1}\ot w_2$. If we act by the element $x$ on $ v^{l} \ot w_1 \we v^{l+2}\ot w_2$ we obtain
  \[ v^{l-1} \ot w_1 \we v^{l+2}\ot w_2 +
  v^{l} \ot w_1 \we v^{l+1}\ot w_2 \in Q.
  \]
 On the other hand, after the action of the element $y$ on  $ v^{l-1} \ot m_1 \we v^{l+1}\ot w_2$ we get
  \[ (k-i-1) v^{l-1} \ot m_1 \we v^{l+2}\ot w_2 +
  (k-i+1) v^{l} \ot w_1 \we v^{l+1}\ot w_2 \in Q.
  \]
Therefore elements  $v^{l-1} \ot w_1 \we v^{l+2}\ot w_2$ and $v^{l} \ot w_1 \we v^{l-1}\ot w_2$ belong to $ Q$.
  
  Now, consider the case $m=2$.  Using the same reasoning with the maps $\alpha_{ij}$ one can show that $\alpha_{ij}=0$ if $i,j\ge3$ or $i,j\le k-3$. Therefore if $k\ge6$ the map $\Sop^2$ is injective. Similar computation shows that $\ker \Sop^2=0$ also for $k=5$ and $k=4$.
\end{proof}

\begin{lem}\label{lem5}
If $m=2,k=3$ then a kernel of the operator $\Sop^2$ is a 1-dimensional space.
\end{lem}
\begin{proof}
Similar to above, we can show that $\ker\Sop^2$ is at most 1-dimensional. To complete the proof we need to find a non-zero element in $\ker\Sop^2$. This can be done by using the split real form of the exceptional Lie algebra $G_2$. 
Namely, mark the root system of $G_2$ as follows:
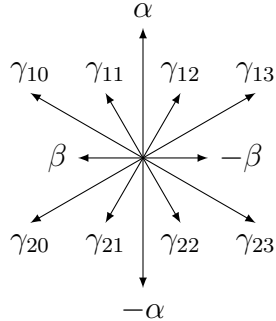
\begin{center}
\begin{tikzpicture}[>=latex]

\draw[->] (0,0) -- ({sqrt(3)*cos(150)},{sqrt(3)* sin(150)}) node[above] {\(\gamma_{10}\)};

\draw[->] (0,0) -- ({cos(120)},{ sin(120)}) node[above] {\(\gamma_{11}\)};

\draw[->] (0,0) -- ({cos(60)},{ sin(60)}) node[above] {\(\gamma_{12}\)};

\draw[->] (0,0) -- ({sqrt(3)*cos(30)},{sqrt(3)* sin(30)}) node[above] {\(\gamma_{13}\)};

\draw[->] (0,0) -- ({sqrt(3)*cos(-150)},{sqrt(3)* sin(-150)}) node[below] {\(\gamma_{20}\)};

\draw[->] (0,0) -- ({cos(-120)},{ sin(-120)}) node[below] {\(\gamma_{21}\)};

\draw[->] (0,0) -- ({cos(-60)},{ sin(-60)}) node[below] {\(\gamma_{22}\)};

\draw[->] (0,0) -- ({sqrt(3)*cos(-30)},{sqrt(3)* sin(-30)}) node[below] {\(\gamma_{23}\)};

\draw[->] (0,0) -- ({0},{sqrt(3)}) node[above] {\(\alpha\)};

\draw[->] (0,0) -- ({0},{-sqrt(3)}) node[below] {\(-\alpha\)};

\draw[->] (0,0) -- ({-sqrt(3)/2},0) node[left] {\(\beta\)};
\draw[->] (0,0) -- ({sqrt(3)/2},0) node[right] {\(-\beta\)};

\end{tikzpicture}
\end{center} 
Associate an element $v^j\ot e_i$, $i=1,2$, $j=0,\dots,3$, with a basis element of the root space $G_2(\gamma_{ij})$, elements $x,y\in \sltw$ with basis elements of $G_2(\beta),G_2(-\beta)$ and elements $e^2_1,e^1_2\in\gl(2,\R)$ with basis elements of $G_2(\alpha),G_2(-\alpha)$. This correspondence can be prolonged to elements $h$ and $e_1^1-e_2^2$ by the formula:
\[ 
h=[x,y]\in [G_2(\beta),G_2(-\beta)],\quad  e_1^1-e_2^2=[e^1_2,e_1^2]\in [G_2(\alpha),G_2(-\alpha)].
\] 
Then a non-zero element $\Phi\in\ker\Sop^2$ corresponds to the restriction of the Lie bracket of $G_2$ to the space $\sum_{i=1}^2\sum_{j=0}^3 G_2(\gamma_{ij})$. The fact that the map $\Phi$ lies in the kernel of $\Sop^2$ follows immediately from Jacobi identities. 
\end{proof}
Note that in this case $\ker \Sop^2$ is a trivial $\sltw$-module, and, hence, all its elements are automatically $x$-invariant.

\subsection{Effective part of the space $E^{0,2}_2$}\label{ss:e02}

The previous subsection gives the full description of the subspace $E^{0,2}_2\subset H^2(\g_{-},\g)$. Namely, we have
\[
E^{0,2}_2 = \Inv_x \left( \frac{\Hom(\wedge^2 V, V)}{\im \Sop^1}\right) \oplus \ker \Sop^2,
\]
where the second summand is non-trivial only for $m=2,k=3$. 

However, not all elements of this subspace correspond to non-vanishing fundamental invariants of the normal Cartan connection. First of all, we need to consider only the elements of positive degree, as regularity condition of the normal Cartan connection implies that the structure function is concentrated in positive degree. However, this is not sufficient 
in order to guarantee that the corresponding invariants are non-trivial. For example, in the next section we show that the fundamental invariant corresponding to $\ker \Sop^2$ (for a system of two 4-th order ODEs) always vanishes. 

On top of it, computing a part of the normal Cartan connection explicitly, we find additional linear conditions on the structure function. Below we compute the space of all elements of $\Inv_x (\Hom(\wedge^2 V, V)/\im \Sop^1)$, which satisfy these additional conditions.

Let, as above, $\gh=\sum_{i\ge 0}\g_i$ be a non-negative part of $\g$. Let $F$ be the following $\gh$-invariant subspace of $V$:
\[
F = \langle v^1, \dots, v^k\rangle \otimes W. 
\]

Define an operator
\begin{equation}\label{l_delta}
\delta\colon \Hom(\we^p F,\R x)\to \Hom(\we^{p+1}F,V/F)=\Hom(\we^{p+1}F,W)
\end{equation}
as
\[
(\delta \om)(A_1,\dots,A_{p+1})=\left(\sum_{i=1}^{p+1}{(-1)}^i\om(A_1,\dots,\hat A_i,\dots,A_{p+1}).A_i \right)\mod F.
\]
This operator plays a crucial role in identifying an effective part of $E^{0,2}_2$.

Identify $V/F$ with $W$ and define the $\gh$-invariant morphisms $i_F \colon F \to V$, $i_{\we^2 F} \colon \we^2 F \to \we^2 V$, $\pi_W \colon V\to V/F$. Further, define  
\[
\al\colon\Hom (\we^2 V,V )\to \Hom (\we^2 F, V/F)
\] 
by composition of $i^*_{\we^2 F}$ with $\pi_W$. The map $\al$ sends every morphism $c\in\Hom(\we^2 V,V)$ to a morphism
\[ 
\al (c)=c|_{\we^2 F}  \mod F.
\]
Since $\al$ is $\gh$-invariant, so is $\ker \al$.

In the next section we explicitly compute coordinate representation of a part of the normal Cartan connection and, in particular, prove the following result.
\begin{prop}\label{prop1}
Let $\om$ be a normal Cartan connection associated with a system of $m\ge 2$ ODEs of order $k+1\ge 4$ and let $c$ be a part of its structure function taking values in $\Hom(\wedge^2 V, V)$. Then $\al(c)$ lies in $\im \delta$.
\end{prop}

Let $\bar \alpha\colon \Hom(V,\sltw)\to\Hom(F,\sltw/\langle h,y \rangle)\cong\Hom(F,\R x)$ be the canonical projection. Then the following diagram is commutative:
\begin{equation}\label{02}
\begin{CD}
  \Hom(V,\sltw)    @>\Sop^1>>  \Hom (\we^2 V,V) \\
@VV\bar\alpha V        @VV\alpha V \\
\Hom(F,\R x)    @>\delta>> \Hom (\we^2 F, W).
\end{CD}
\end{equation}
In particular, we see that the elements from $\im \Sop^1\subset \Hom(\wedge^2 V, V)$ automatically satisfy the condition from Proposition~\ref{prop1}. 

Let us now describe all elements of $\Inv_x (\Hom (\we^2 V,V)/ \im \Sop^1)$ whose representatives $c\in \Hom(\wedge^2 V, V)$ satisfy the condition from Proposition~\ref{prop1}.

\begin{thm}\label{thm:E02} 
For $k\ge 3$ the the space
\[ 
\{ [c]\in \Inv_x (\Hom_{+}(\wedge^2 V,V)/\im \Sop^1) \mid \alpha(c)\in \im \delta\}
\]
is concentrated in degree 2 and is isomorphic (as a $\glm$-module) to $S^2(W^*)\ot W$.
\end{thm}
\begin{proof}
We shall need the following additonal result.
\begin{lem}\label{l3}
The restriction $\pi^x_W$ of the map $\pi_W$ to  $\Inv_x \Hom (\we^2 V,V)$ is injective.
Moreover the image of the map $\pi^x_W$ is equal to \[\ker x^{k+1}|_{\Hom(\wedge^2 V,W)}.\]
\end{lem}
\begin{proof}
  Indeed, suppose $c\in \Inv_x\Hom(\wedge^2 V, V)$. Decompose $c$ as 
\[
c = \sum_{i=0}^k c_i e_i,
\]
where $c_i \in \Hom(\wedge^2 V, v^i\otimes W)$. Then we have
\[
x.c = \sum_{i=0}^k(x.c_i)e_i +
\sum_{i=0}^{k-1}c_{i+1} e_i = 0.
\]
Hence, $c_i = - x.c_{i-1}$ for all $i=1,\dots,k$. Therefore,
$c_0=0$ implies that $c_i=0$ for all $i>0$ and, thus,
$c = 0$.

The second part of the lemma follows directly from the equality
\[
x.c_k = (-1)^{k+1}x^{k+1}.c_0 = 0.
\]
\end{proof}

From~\eqref{02} it follows that for any $[c]\in \Inv_x (\Hom(\wedge^2 V,V)/\im \Sop^1)$ satisfying $\alpha(c)\in \im \delta$ we can assume that $\alpha(c)=0$ and $x.c=0$. By Lemma~\ref{l3} we need to compute the space:
\begin{align*} T&=\Inv_{x^{k+1}}\pi_W(\ker\alpha) \\
&= \Inv_{x^{k+1}} \left( \Hom(\we^2 W+W\ot F,W) \right)
\end{align*}
modulo $\pi_W(\Inv_x \im \Sop^1 \cap \ker \alpha)$.

All elements from $\Hom(\we^2 W+W\ot F,W)$ have the following form:
\[v={v^0}^*e_{i_1}^*\we {v^i}^*e_{i_2}^*\ot A_i^{i_1,i_2}+
{v^1}^*e_{i_1}^*\we {v^{i-1}}^*e_{i_2}^*\ot\beta_i^{i_2} e_{i_1}. \]
Then the action of $x^{k+1}$ on $v$ is
\begin{align*}
x^{k+1}.v&=(-1)^{k+1}\sum_{j=0}^{k+1} C_{k+1}^j 
{v^j}^*e_{i_1}^*\we {v^{i+k+1-j}}^*e_{i_2}^*\ot A_i^{i_1,i_2}\\
&+\sum_{j=1}^{k+1} C_{k+1}^{j-1} 
{v^j}^*e_{i_1}^*\we {v^{i+k+1-j}}^*e_{i_2}^*\ot \beta_i^{i_2} e_{i_1}
\end{align*}
Decompose the right hand side in the standard basis of $\we^2 V^*$. Then from $x^{k+1}.v=0$ we get:
\begin{itemize}
\item for $i<j<\frac{i+k+1}{2}$ a coefficient at ${v^j}^*e_{i_1}^*\we {v^{i+k+1-j}}^*e_{i_2}^*$ gives:
\begin{equation}\label{eq7} C^j_{k+1}A^{i_1,i_2}_i- C^{i+k+1-j}_{k+1}A^{i_2,i_1}_i+
 C_{k+1}^{j-1}\beta_i^{i_2} e_{i_1} -
 C_{k+1}^{i+k-j}\beta_i^{i_1} e_{i_2}=0;
\end{equation}
\item for $j=\frac{i+k+1}{2}$ a coefficient at ${v^j}^*e_{i_1}^*\we {v^{j}}^*e_{i_2}^*$ gives:
\begin{equation}\label{eq8} C^j_{k+1}(A^{i_1,i_2}_i-A^{i_2,i_1}_i)+
 C_{k+1}^{j-1}(\beta_i^{i_2} e_{i_1}-\beta_i^{i_1} e_{i_2}) =0;\end{equation}
\end{itemize}

Equations (\ref{eq7}-\ref{eq8}) imply that for every $\beta_i$, $i<k-1$ there exists at most one tensor $A_i$ which satisfies equations (\ref{eq7}-\ref{eq8}). Moreover if $\beta_i=0$ then $A_i$ should be zero. Using equation~\eqref{eq8} we conclude the same for $\beta_{k-1}$ and for the antisymmetric part of the tensor $A_{k-1}$.

This can be reformulated as follows. The space $\we^2 V^*$ is decomposed into the direct sum 
\[ 
\we^2 V^* = \we^2 V_k^*\ot S^2(W^*)+ S^2(V^*_k)\ot\we^2 W^*.
\]
Using this decomposition, we can check directly that elements from $ {v^0}^*\ot {v^k}^*\ot W^*\ot W^*\ot W$ and  $\langle v^0\we v^{k-1}\rangle^*\ot S^2W^*\ot W$ are $x^{k+1}$-invariant and belong to $\ker\alpha$. Then the space $\Inv_{x^{k+1}}\pi_W(\ker\alpha)$ is equal to
\[ T={v^0}^*\ot {v^k}^*\ot W^*\ot W^*\ot W +\langle v^0\we v^{k-1}\rangle^*\ot S^2W^*\ot W.
\]

It remains to factor the space $T$ by elements from $\pi_W(\Inv_x \im \Sop^1)$ lying in $\ker \alpha$.  Note that $\pi_W\Sop^1\left(\langle {v^k}^* \rangle\ot W^*\ot\glm\right)$ is $x$-invariant, belongs to $\ker \alpha$ and is equal to:
\[ 
{v^0}^*\ot {v^k}^*\ot W^*\ot W^*\ot W. 
\]
Therefore, we can assume that 
\[
T\subset \langle v^0\we v^{k-1}\rangle^*\ot S^2W^*\ot W 
\]
and is concentrated in degree 2. The space of all degree 2 elements in $\Inv_x\Hom(V,\sltw)$ is generated by 
\[
\om_i= 2 ( v^{k-2} \ot e_i)^*\ot x + (k-1)(  v^{k-1} \ot e_i)^*\ot h + k(k-1)( v^k \ot e_i)^*\ot y.
\]
$i=1,\dots,m$. However, for $k\ge 3$ any non-zero linear combination of $\pi_W\Sop^1(\om_i)$ does not lie in $\ker\alpha$.
\end{proof}
The explicit formulas for the corresponding fundamental invariant are computed in the next section and given in Theorem~\ref{thm_I2}.

\subsection{The structure of the space $E^{1,1}_2$}
As shown in~\cite{dou01} for the case of a scalar ODE, this part of the cohomology space corresponds to generalized Wilczynski invariants for the systems of ODEs. For completeness we provide a purely algebraic description of this space. 

According to the Proposition~\ref{l1}, the space $E^{1,1}_2$ is isomorphic to 
\[ 
H^1(\R x,\Hom(V,V)/\im \Sop^0),
\]
where $\Sop^0$ is an inclusion of Lie algebra $\ga$ into $\gl(V)$. Lemma~\ref{l2} implies that we should describe the set of $y$-invariant elements in $\sltw$-module $\gl(V)/\ga$.

The $\ga$-module $\gl(V)$ is isomorphic to $\gl(V_k)\ot\gl(W).$ Let's identify Lie algebra $\ga$ with its image in $\gl(V).$ In particular, we denote as $y$ the image of $y\in\sltw$ in $\gl(V).$ 
\begin{thm}
The space of $y$-invariant elements in the $\ga$-module $\gl(V)/\ga$ is the sum of the following $\glm$-modules:
\begin{align*}
A_2 &=\R y\ot \gsl(W)\\
A_{i+1}&=\R y^{i}\ot \gl(W),\quad i=2,\dots,k
\end{align*}
For every element $\phi\in A_i$ the corresponding element $c_\phi\in E^{1,1}_2$ of the form $c_\phi\colon x\to\phi$ has degree $i$.
\end{thm}
\begin{proof}
The decomposition of the $\sltw$-module $\gl(V_k)$ is well known:
\[ 
\gl(V_k)=V_0\op V_2\op\cdots\op V_{2k} .
\]
Note that all endomorphisms $y^i$, $i=0,\dots,k$ are $y$-invariant and linearly independent. Therefore the space of $y$-invariant elements is the sum of submodules $\R y^i\ot\gl(W)$ for $i=0,\dots,k$. 

The space of $y$-invariant elements in $\ga$ is nothing else but $\R y\op\glm$. It is not hard to see that under the natural inclusion $\ga\hookrightarrow\gl(V)$ the space $\glm$ goes to $\R y^0\ot\gl(W)$ and $\R y$ goes to  $\R y\ot \Id_W$. Finally, we can identify the space $ \R y\ot\gl(W)/\R y\ot\Id_W$ with $ \R y\ot\gsl(W)$.
\end{proof}
Each of the submodules $A_i$, $i=2,\dots,k$ corresponds to the generalized Wilczynski invariant $W_i$ of degree $i$.

\section{Parametric computation of the normal Cartan connection}\label{sec:param}
In this section we provide an explicit formula for the invariant $I_2$, which is described by Theorem~\ref{thm:E02},  give the proof to Proposition~\ref{prop1} and show that the part $\ker \Sop^2$ of $E^{0,2}_2$ in case $m=2$, $k=3$ does not produce any non-trivial invariants.

For every regular Cartan connection adapted to the equation~\eqref{eq1} we can choose a section $s\colon\E\to\cG$ such that the pullback of the connection form to~$\E$ is:
 \[
 \omega =\sum_{r=1}^{k+1}\omega^i_{-r} v^{k+1-r}\otimes e_i + \omega_x x  + \omega_h h +
\omega^i_j e^j_i + \omega_y y ,
\]
where
\begin{align*}
\omega_{-r}^i &=\theta_{-r}^i + \sum_{s=r+1}^{k+1}A_{j,-r}^{i,-s}\theta_{-s}^j, 
\\
\omega_x &=-\theta_x + \sum_{s=2}^{k+1}B_j^{-s}\theta_{-s}^j ,
\\
\omega_h &=\sum_{s=1}^{k+1}C_j^{-s}\theta_{-s}^j,
\\
\omega^i_j & =D^{i,x}_j \theta_x+\sum_{s=1}^{k+1}D_{j,l}^{i,-s}\theta_{-s}^l ,
\\
\omega_y & =E^x \theta_x +  \sum_{s=1}^{k+1}E_j^{-s}\theta_{-s}^j.
\end{align*}
and the forms $\th_x$, $\th^i_{-r}$ are given by:
\begin{align*}
  \theta_x & =dx, \\
  \theta^i_{-1} & =dy^i_k-f^i\,dx,&\\
  \theta^i_{-r} & =dy^i_{k+1-r}-y^i_{k+2-r}\,dx, \quad 2\le r\le k+1.
\end{align*}

Let $\Omega$ be the curvature tensor of~$\omega $. We use the following notation:
\[
\Omega  =\sum_{r=1}^{k+1}\Omega^i_{-r} v^{k+1-r}\otimes e_i + \Omega_x x + \Omega_h h + \Omega_i^j e_j^i + \Omega_y y 
\]

\subsection{Proof of Proposition~\ref{prop1}}\label{ss:thm1}
The $\Hom(\wedge^2 F, W)$ part  of the curvature function $c$ corresponds to the coefficients of $\Omega^i_{-k-1}$ at $\omega^{j_1}_{-s_1}\wedge \omega^{j_2}_{-s_2}$, $s_1,s_2=2,\dots,k+1$. We have:
\begin{multline*}
\Om^i_{-k-1} = \, d \omega_{-k-1}^i + \omega_x \wedge \omega_{-k}^i + k\omega_h \wedge
\omega_{-k-1}^i  +\omega_j^i \wedge \omega_{-k-1}^j = \\
d\om^i_{-k-1}+\om_x\wedge \om^i_{-k} \mod \langle \om^p_{-k-1} \rangle = \\
d\th^i_{-k-1}+\om_x\wedge\om^i_{-k} = \th_x \wedge \th^i_{-k} + \om_x\wedge\om^i_{-k} \mod \langle \om^p_{-k-1} \rangle.
\end{multline*}
Using the fact that $\om^i_{-k} = \th^i_{-k} \mod \langle \om^p_{-k-1} \rangle$ and that
\[
\th_x = -\om_x + \sum_{s=2}^{k+1} \bar B^{-s}_j \om^j_{-s}
\]
for some functions $\bar B^{-s}_j$ (expressed polynomially through $B^{-s}_j$ and $A^{i,-s}_{j,-r}$), we further get:
\begin{multline}\label{eq:Omk1}
\Om^i_{-k-1} = (-\om_x + \sum_{s=2}^{k+1} \bar B^{-s}_j \om^j_{-s})\wedge \om^i_{-k} + \om_x\wedge\om^i_{-k} = \\
\sum_{s=2}^{k} \bar B^{-s}_j \om^j_{-s}\wedge \om^i_{-k} \mod \langle \om^p_{-k-1} \rangle.
\end{multline}
Define the map $\alpha\in \Hom(F,U)$ by:
\begin{align*}
&\alpha\colon v^{k+1-s}\otimes e_j \mapsto -\bar B^{-s}_j x,\quad s=2,\dots,k.
\end{align*}
Then equation~\eqref{eq:Omk1} means that the $\Hom(\wedge^2 F, W)$-part of the structure function $c$ is equal exactly to $\delta(\alpha)$. In particular, we see that the component of $\gamma(c)$ lying in $\Hom(\wedge^2 F, W)/\delta \left(\Hom(F,U)\right)$ vanishes identically for any regular Cartan connection associated with system~\eqref{eq1}.

\subsection{Computation of $I_2$}
To find the explicit expression for the fundamental invariant prescribed by Theorem~\ref{thm:E02} it is sufficient to compute the curvature of the normal Cartan connection up to degree~$2$. 

There are no fundamental invariants in degree $1$. So, all coefficients of $\Omega$ in this degree should vanish. We have:
\begin{align*}
\Omega_x&=\,d \omega _x + 2\omega _h \wedge \omega_x\equiv B_j^{-2}\theta_x\we\theta_{-1}^j
\\
&+2C_j^{-2}\theta_{-1}^j\we\theta_x  \mod \left\langle \om_{-2}^p,\dots,\om_{-k-1}^p \right\rangle ,
\\
\Omega_{-1}^i&=\,d \omega_{-1}^i  - k \omega_h \wedge
\omega_{-1}^i+ k \omega_y \wedge
\omega_{-2}^i + \omega_j^i \wedge \omega_{-1}^j 
\\
&\equiv \frac{\p f^i}{\p y_k^j}\theta_x\we\theta_{-1}^j+A^{i,-2}_{j,-1}\theta_x\we\theta_{-1}^j-k C^{-1}_l\theta_{-1}^l\we \theta_{-1}^i
\\
&+ D^{i,x}_{j}\theta_x\we\theta_{-1}^j+ D^{i,-1}_{j,l}\theta_{-1}^l\we \theta_{-1}^j
\mod \left\langle \om_{-2}^p,\dots,\om_{-k-1}^p \ \right\rangle ,
\\
\Omega_{-r}^i & =\, d \omega_{-r}^i + \omega_x \wedge \omega_{-r+1}^i + (2r-k-2)\omega_h \wedge
\omega_{-r}^i 
\\
 & +(k+1-r)r\,\omega_y \wedge \omega_{-r-1}^i
 +\omega_j^i \wedge \omega_{-r}^j 
\\
& \equiv (A^{i,-r-1}_{j,-r}-A^{i,-r}_{j,-r+1}+D^{i,x}_{j})\theta_x\we\theta_{-r}^j
\\
& + (2r-k-2)C^{-1}_{l} \theta_{-1}^l\we \theta_{-r}^i
 + D^{i,-1}_{j,l}\theta_{-1}^l\we \theta_{-r}^j
 \\ &
\mod \left\langle \omega_{-s_1}^{p_1}\we \omega_{-s_1}^{p_1} \,|\, s_1+s_2 > r+1 \right\rangle, \quad r=2,\dots,k,
\\
\Omega_{-k-1}^i & =\, d \omega_{-k-1}^i + \omega_x \wedge \omega_{-k}^i + k\omega_h \wedge
\omega_{-k-1}^i  +\omega_j^i \wedge \omega_{-k-1}^j 
 \\
 &\equiv (-A^{i,-k-1}_{j,-k}+D^{i,x}_{j})\theta_x\we\theta_{-k-1}^j + kC^{-1}_{l} \theta_{-1}^l\we \theta_{-k-1}^i + \\
 & D^{i,-1}_{j,l}\theta_{-1}^l\we \theta_{-k-1}^j 
\mod \left\langle \omega_{-s_1}^{p_1}\we \omega_{-s_1}^{p_1} \,|\, s_1+s_2 > k+2 \right\rangle.
\end{align*}
Setting the terms above equal to $0$ we get:
\begin{align*}
& C_j^{-1}=B_j^{-2}=D_{j,l}^{i,-1}=0,
\\
& D^{i,x}_j=-\frac{1}{k+1}\frac{\p f^i}{\p y^j_k},\quad
A^{i,-r-1}_{j,-r}=-\frac{(k+1-r)}{k+1}\frac{\p f^i}{\p y^j_k},
\end{align*}
where $r=1,\dots,k$.

Now proceed to the second degree. We consider only $\Hom(\wedge^2 V,V)$ part of the curvature. This means  that we want to compute only the coefficients of $\Omega^i_{-r}$ at $\om_{-r_1}^i\we\om_{-r_2}^j.$ Modulo $\left\langle \omega_x , \omega_{-s_1}^{p_1}\we \omega_{-s_1}^{p_1} \,|\, s_1+s_2 > k+2 \right\rangle$ we have:
\begin{align*}
\Omega_{-1}^i &\equiv \left(-kC_l^{-2}\delta_j^i+D^{i,-2}_{j,l}-\frac{\p A^{i,-2}_{l,-1}}{\p y_k^j}-k\delta^i_l E_j^{-1}\right)\theta_{-2}^l\we\theta_{-1}^j,
\\
\Omega_{-2}^i &\equiv \left(-(k-2)(C_l^{-2}\delta_j^i-C_j^{-2}\delta_l^i)+D^{i,-2}_{j,l}-D^{i,-2}_{l,}\right)\theta_{-2}^l\we \theta_{-2}^j
\\
&+\left(\frac{\p A^{i,-3}_{j,-2}}{\p y_k^l}+2(k-1)\delta_j^iE^{-1}_l-\delta_l^i B^{-3}_j\right)\theta_{-1}^l\we\theta_{-3}^j 
\\
\Omega_{-3}^i &\equiv \left(-(k-4)C_l^{-2}\delta_j^i+D^{i,-2}_{j,l}-\delta_l^i B^{-3}_j\right)\theta_{-2}^l\we \theta_{-3}^j
\\
&+\left(\frac{\p A^{i,-4}_{j,-3}}{\p y_k^l}+3(k-2)\delta_j^iE^{-1}_l\right)\theta_{-1}^l\we\theta_{-4}^j 
\\
\Omega_{-r}^i &\equiv B^{-3}_l\theta_{-3}^l\we \theta_{-r+1}^i+\left((2r-k-2)C_l^{-2}\delta_j^i+D^{i,-2}_{j,l}\right)\theta_{-2}^l\we \theta_{-r}^j
\\
&+\left(\frac{\p A^{i,-r-1}_{j,-r}}{\p y_k^l}+r(k+1-r)\delta_j^iE^{-1}_l\right)\theta_{-1}^l\we\theta_{-r-1}^j,\, r=4,\dots,k,
\\
\Omega_{-k-1}^i &\equiv B^{-3}_l\theta_{-3}^l\we \theta_{-k}^i+ \left(kC_l^{-2}\delta_j^i+D^{i,-2}_{j,l}\right)\theta_{-2}^l \we\theta_{-k-1}^j.
\end{align*}

Denote coefficients of the curvature function at $\om^i_{-r}\we\om^j_{-s}$ and $\om^i_{-r}\we\om_x$ by $\Om[\om^i_{-r}\we\om^j_{-s}]$ and $\Om[\om^i_{-r}\we\om_x]$ respectively. Canonical normalization conditions in the second degree are given by:
\begin{align}
\label{e_s}\sum_{i;r=2..k+1}\Om^i_{-r+1}[\om_{-1}^l\we\om_{-r}^i]-2\Om_h[\om_{-1}^l\we\om_x]=0,
\\
\sum_{i;r=1..k+1} \Om^i_{-r}[\om_{-2}^l\we\om_{-r}^i]-k\Om^j_{i}[\om_{-1}^l\we\om_x]=0,
\\
\begin{aligned}\sum_{i;r=1..k+1} (2r-k-2)\Om^i_{-r}[\om_{-2}^l\we\om_{-r}^i]&+\\+2k\Om_{h}[\om_{-1}^l\we\om_x]&+2\Om_{x}[\om_{-2}^l\we\om_x]=0,
\end{aligned}
\\
\label{e_e}\sum_{i;r=1..k} r(k-r-1)\Om^i_{-r-1}[\om_{-3}^l\we\om_{-r}^i]+2(k-1)\Om_{x}[\om_{-2}^l\we\om_x]=0,
\\
\frac{1}{(k-r)(r+1)}\Om^i_{-r-1}[\om_x\we\om_{-r-2}^j]-\frac{1}{(k+1-r)r}\Om_{-r}^i[\om_x\we\om_{-r-1}^j]=0,
\\
\sum_{i;r=2..k+1}\Om^i_{-r+1}[\om_{-r}^i\we\om_x]=0.
\end{align}
where 
\begin{align*}
\Om^i_{j}[\om_{-1}^l\we\om_x]&=D^{i,-2}_{j,k}-\frac{\p D^{i,x}_j}{\p y^l_k}, 
\\
\Om_{h}[\om_{-1}^l\we\om_x]&=C^{-2}_l-E_l^{-1},
\\
\Om_{x}[\om_{-2}^l\we\om_x]&=B^{-3}_l+2C_l^{-2}.
\end{align*}
The solution of equations \eqref{e_s}-\eqref{e_e} is:
\begin{align}
D_{j,l}^{i,-2}&=-\frac{1}{k+1}\frac{\p^2 f^i }{\p y_k^j \p y_k^l}
\\
\label{B3}
E_l^{-1}=C_l^{-2}=-\frac12 B_l^{-3}&=\frac{3(k-1)}{(k^2 m+k m+6)(k+1)}\frac{\p}{\p y^l_k}\sum_{i=1}^m \frac{\p f^i }{\p y_k^j} 
\end{align}

From solutions above we can conclude that $\Hom(\we^2 V,V)$ part of the curvature function depends only on coefficients of  tensor 
\begin{equation}\label{I2}
\frac{\p^2 f^i }{\p y_k^j \p y_k^l}.
\end{equation}
Moreover, the whole tensor \eqref{I2} is a relative invariant of the structure since: 
\begin{align*} 
\Omega_{-1}^i[\om_{-2}^l\we\om_{-1}^j]+\Omega_{-k-1}^i[\om_{-2}^j\we\om_{-k-1}^l]+\Omega_{-k-1}^i[\om_{-2}^l\we\om_{-k-1}^j]
\\
=2D^{i,-2}_{j,l}=-\frac{2}{k+1}\frac{\p^2 f^i }{\p y_k^j \p y_k^l}.
\end{align*}

We summarize the computations of this section in the following theorem.
\begin{thm}\label{thm_I2}
For the system of ODEs~\eqref{eq1} of the order $\ge 3$ the following symmetric tensor is a relative 
differential invariant of degree $2$:
\[
(I_2)^i_{jl}=\frac{\p^2 f^i }{\p y_k^j \p y_k^l}.
\]
\end{thm}

\subsection{Vanishing of the invariant corresponding to $\ker \Sop^2$ for systems of two ODEs of 4-th order}
According to the proof of Lemma~\ref{lem5}, the fundamental invariant $I$ corresponding to $\ker\Sop^2$ 
appears in the component $\Om_x[\om_{-3}^1\we\om_{-3}^2]$ of the curvature function function. Namely, this coefficient is equal to $I$ modulo fundamental invariants of lower order and their covariant derivatives.

We are going to show that $\Om_x[\om_{-3}^1\we\om_{-3}^2]$ for the normal Cartan connection belongs to differential ideal generated by $I_2$ and, therefore, the generator of $\ker\Sop^2$  does not corresponds to any new fundamental invariants. We have:
\[ 
\Omega_x=\,d \omega _x + 2\omega _h \wedge \omega_x\equiv d(-\theta_x + B_j^{-3}\theta_{-3}^j + B_j^{-4}\theta_{-4}^j)\mod \left\langle \om_x \right\rangle.
\]
Modulo $ \left\langle \om_x,\om_{-1}^s,\om_{-2}^s,\om_{-4}^s \right\rangle $ we get:
\begin{multline*}  
d( B_j^{-3}\theta_{-3}^j )=\frac{\p B_j^{-3}}{\p \om_{-r}^i} \om_{-r}^i\we \theta_{-3}^j+\frac{\p B_j^{-3}}{\p \om_x} \om_x\we \theta_{-3}^j+ B_j^{-3}\theta_x\we\theta_{-2}^j
\\
\equiv \left(\frac{\p B_2^{-3}}{\p \om_{-3}^1}-\frac{\p B_1^{-3}}{\p \om_{-3}^2} \right) \om_{-3}^1\we \om_{-3}^2+ B_j^{-3}A_{i,-2}^{j,-3}B^{-3}_l\om_{-3}^i\we \om_{-3}^l,
\end{multline*}
where $\frac{\p}{\p \om_{-r}^i}$ are covariant derivatives along fundamental vector fields dual to $\om_{-r}^i $. 

From~\eqref{B3} it follows that $B^{-3}_j$ is expressed linearly in terms of components of the fundamental invariant~$I_2$. Therefore $\Om_x[\om_{-3}^1\we\om_{-3}^2]$ belongs to differential ideal generated by $I_2$, and  the generator of $\ker\Sop^2$ does not lead to any new fundamental invariants.

\section{Cohomology of contact Lie algebras of higher order}\label{coh_M}
As shown in~\cite{mor88, dkm}, the Tanaka symbol of the contact distribution on $J^{k}(\R,\R^m)$, can be identified with the Lie algebra of polynomial vector fields on $\R^{m+1}$, equipped with a grading that depends on $k$. More precisely, let $U$ be a one-dimensional vector space spanned by $x$ and $W$ be an $m$-dimensional vector space spanned by $v^0\otimes e_i$, $i=1,\dots,m$. In particular, $U$ has degree $-1$, and elements of $W$ have degree $-k-1$. Then the symbol algebra of the contact distribution on $J^{k}(\R,\R^m)$ is 
\[
\bar \g = S(U^*\op W^*)\ot(U\op W).
\]

The Lie algebra $\g$ is, obviously, included into the Lie algebra $\bar\g$. Additionally we should note that $\g_-=\bar\g_-$. The inclusion above induces the map between cohomologies. We are going to study this map in more detail.

Let $V=\op_{i=0}^kS^i(U^*)\ot W = V_k \ot W$ and 
\[
F=\op_{i=1}^kS^i(U^*)\ot W=\langle v^1 ,\dots, v^k\rangle\ot W. 
\]
Similar to~\eqref{l_delta} we define the linear operator
\[ 
\delta\colon \Hom(\we^p F,U)\to \Hom(\we^{p+1}F,V/F)=\Hom(\we^{p+1}F,W)
\]
\begin{equation}\label{l_delta2}
(\delta \om)(A_1,\dots,A_{p+1})=\left(\sum_{i=1}^{p+1}{(-1)}^i\om(A_1,\dots,\hat A_i,\dots,A_{p+1}).A_i \right)/F.
\end{equation}
In this setting the structure of the cohomology space $H(\bar \g_-,\bar \g)$ is described by the following result of Morimoto~\cite{mor88}.

\begin{thm*}[Morimoto]
One has the following long exact sequence:
\begin{multline*}\cdots\to\Hom(\we^p F,W)\to H^p(\bar \g_-,\bar \g) 
\\
\to \Hom(\we^p F,U)
\overset{\delta}\to\Hom(\we^{p+1} F,W)\to\cdots  .
\end{multline*}
\end{thm*}
As an immediate corollary of this theorem we obtain that $H^1(\bar \g_-,\bar\g)=0$ for $m\ge 2$. Thus, we have the following exact sequence:
\begin{multline*}
0\to\Hom(F,U)\overset{\delta}\to\Hom(\we^2 F,W)\to H^2(\bar \g_-,\bar \g) 
\\
\to \Hom(\we^2 F,U)\overset{\delta}\to\Hom(\we^{3} F,W). 
\end{multline*}
We see that $H^2(\bar \g_-,\bar\g)$ consists of $2$ parts: 
\[ \Hom(\we^2 F,W)/\delta\Hom(F,U)\] and 
\[\ker\left(\delta\colon \Hom(\we^2 F,U)\to\Hom(\we^{3} F,W)\right).\]

Consider now a map 
\[
\gamma\colon H^2(\g_-,\g) \to H^2(\bar \g_-,\bar\g)
\]
which is induced by the inclusion of Lie algebras $\g\hookrightarrow\bar\g$. 
 
\begin{thm}\label{thm:gamma}
The invariants of the normal Cartan connection corresponding to the part of curvature function taking values in $\ker\gamma \cap H^2_{+}(\g_{-},\g)\subset H^2(\g_{-},\g)$ form the complete system of fundamental invariants.
\end{thm}
\begin{proof}
From the above description of $H^2(\g_{-},\bar\g)$ it follows immediately that $\gamma$ vanishes identically on the space $E^{1,1}_2$, which corresponds to generalized Wilczynski invariants. Indeed, we can choose representatives of $E^{1,1}_2$ from elements of $\Hom(\R x \otimes V, V)$. But all these elements are mapped to $0$ by $\gamma$.

So, we need to consider only the restriction of the map $\gamma$ to the space $E^{0,2}_2=\Inv_x H^2(V,\g)$. By Theorem~\ref{thm:E02}, the effective part of $E^{0,2}_2$ corresponds to elements $[c]\in\Inv_x(\Hom_{+}(\wedge^2V,V)/\im\Sop^1)$ satisfying the inclusion $\alpha(c)\in\im \delta$ (see Proposition~\ref{prop1}). But it is exactly the reformulation of the statement that $\gamma([c]) = 0$. 

It is also easy to see that the restriction of $\gamma$ to $\ker\Sop^2\subset E^{0,2}_2$ for $m=2,k=3$ is non-zero. Indeed, $\Hom(\we^2 F, U)$ component of this map sends $v^1\ot e_1\we v^1\ot e_2$ to $\R x$ and this part belongs to the $\ker\delta$. Therefore, $\gamma(\Phi)\ne 0$ for any non-zero element $\Phi\in \ker \Sop^2$.
\end{proof}

We note that the condition $\alpha(c)\in\im \delta$ comes from Proposition~\ref{prop1}, which is based on parametric computation of the normal Cartan connection. It would be natural to expect that there is also a coordinate-free proof of the above Theorem. But this requires the development of the theory of (normal) Cartan connections with infinite-dimensional structure group and lies beyond the scope of this paper.

\end{document}